\documentclass[11pt]{amsart}
\usepackage{amscd, amsmath, amsthm, amssymb, xy, xypic, multirow}
\newtheorem{theorem}{Theorem}[section]
\newtheorem{lemma}[theorem]{Lemma}
\newtheorem{proposition}[theorem]{Proposition}

\theoremstyle{definition}     
\newtheorem{definition}[theorem]{Definition}

\numberwithin{equation}{section}

\def \hd #1 {\bfseries #1  \mdseries}
\def \italic #1 {\bfseries \it #1 \rm \mdseries}
\def \ra {\rightarrow}
\def \cen #1 { \begin{center} #1 \end{center}}

\def \mbz {\mathbb Z}

\def \mbp {\mathbb P}

\def \mbq  {\mathbb {Q}}
\def \mco  {\mathcal {O}}

\def \Q {${\mathbb {Q}}\,$}

\def \Pic {{\rm{Pic}}}

\def \Sing {{\rm{Sing}}}

\begin{document}
\title
{Calabi--Yau double coverings of Fano--Enriques threefolds}

\author{Nam-Hoon Lee}
\address{
Department of Mathematics Education, Hongik University
42-1, Sangsu-Dong, Mapo-Gu, Seoul 121-791, Korea
}
\email{nhlee@hongik.ac.kr}
\address{School of Mathematics, Korea Institute for Advanced Study, Dongdaemun-gu, Seoul 130-722, South Korea }
\email{nhlee@kias.re.kr}
\begin{abstract}This note is a report on the observation that  the Enriques--Fano threefolds with  terminal cyclic quotient
singularities admit
Calabi--Yau threefolds as their double coverings. We calculate the invariants of those Calabi--Yau threefolds when the Picard number is one. It turns out that all of them are new examples.
\end{abstract}
\maketitle
\section{Introduction}
The threefolds whose hyperplane sections are Enriques surfaces were studied by G.\ Fano in a famous paper \cite{Fa}. The modern proofs for the results of  \cite{Fa} were given in \cite{CoMu}.
Such varieties are always singular and their canonical divisors are not Cartier but numerically equivalent to Cartier divisors. We call such threefolds \emph{Fano--Enriques threefolds} (see Definition \ref{def1}).
In this note, we consider Fano--Enriques threefolds whose singularities are  terminal cyclic quotient
ones. It is worth noting that any
Fano--Enriques threefold with  terminal singularities admits a \Q-smoothing to one with terminal cyclic quotient
singularities (\cite{Mi}).
The canonical coverings (which are double-coverings) of Enriques--Fano threefolds with terminal cyclic quotient singularities are smooth Fano threefolds (\cite{Ba, Sa}).  Hence all the singular points of  such Enriques--Fano threefolds are of type $\frac{1}{2} (1, 1, 1)$. Using the classification of smooth Fano threefolds,  L.\  Bayle (\cite{Ba}) and T.\ Sano (\cite{Sa}) gave a classification of such threefolds.
In this note, we observe that all those Enriques--Fano threefolds also admit some Calabi--Yau threefolds as their double covering, branched along some smooth surfaces and those eight singularities.
A \italic{Calabi--Yau threefold $Y$} is a compact K\"ahler manifold with
trivial canonical class such that the intermediate cohomology groups of
its structure sheaf are trivial  ($h^1 (Y, \mco_Y) = h^2(Y, \mco_Y) =0$).
We calculate the invariants of those Calabi--Yau double coverings when their Picard numbers are one
(Table \ref{t1}).
It turns out that all those Calabi--Yau threefolds are new examples.
Although a number of Calabi--Yau threefolds have been constructed, those with Picard number one are still
quite rare. Note that they are primitive and play
an important role in the moduli spaces of all Calabi--Yau threefolds (\cite{Gr}).

\section{Calabi--Yau double coverings}
As the higher dimensional algebraic geometry being developed, the definition of Fano--Enriques threefolds also has evolved and has been generalized.
We adapt the following version of the definition.
\begin{definition} \label{def1}
A three-dimensional normal projective variety $W$ is called a Fano--Enriques threefold if $W$ has canonical singularities, $-K_W$ is not a Cartier divisor but numerically equivalent to an ample Cartier divisor $H_W$.
\end{definition}
Y.\ Prokhorov proved in \cite {Pr} that the generic surface in the linear system $|H_W|$ is an Enriques surface with canonical singularities and that the Enriques surface is smooth if the singularities of $W$ is isolated and $-K_W^2 \neq 2$. We refer to \cite {Ch1, GLMR, Pr1} for more systematic expositions of Fano--Enriques threefolds.
In this note, we consider the case that $W$ has only terminal cyclic quotient singularities.
We summarize the properties of $W$ (\cite{Ba, CoMu, Fa, Sa}).
\begin{enumerate}
\item All the singularities of $W$  are  the type of $\frac{1}{2}(1,1,1)$.
\item The number of singularities of $W$ is eight.
\item $-2K_W$ is linearly equivalent to $-2H_W$.
\item There is a smooth Fano threefold that covers doubly $W$, branched only at the singularities of $W$.
\end{enumerate}
L.\ Bayle (\cite{Ba}) and T.\ Sano (\cite{Sa}) gave a classification of  smooth Fano threefolds that double cover Fano--Enriques threefolds.

Let  $\varphi:X \ra W$ be the double covering, branched along the singularities of $W$. Then $X$ is one of smooth Fano threefolds in Theorem 1.1 in \cite{Sa}.
We want to find a Calabi--Yau threefold that  double-covers $W$, using the following theorem which is a special case of Theorem 1.1 in \cite{Lee1}.
\begin{theorem} \label{leethm}
Let $W$ be a projective three-dimensional variety with singularities of type $\frac{1}{2}(1,1,1)$ such that $h^1(W, \mco_W) = h^2(W, \mco_W)=0$. Suppose that the linear system $|-2K_W|$ contains a smooth surface $S$, then there is Calabi--Yau threefold $Y$ that is a double covering of $W$ with the branch locus $S \cup \Sing(W)$.
\end{theorem}

Let $p_1, p_2$ be any points of $X$.
From the description of those Fano threefolds $X$'s in Theorem 1.1 of \cite{Sa},
one can find  an effective divisor $D$ from $|-K_X|$ such that $D$ does not contain $p_1, p_2$.
Let $\theta$ be the covering involution on $X$, i.e.\ the quotient $X/\langle \theta \rangle $ is  $W$.
Therefore, for any point $q \in W$ and we can find an effective divisor $D$ in the linear system $|-K_X|$ such that $D \cap \varphi^{-1}(\{q \}) = \emptyset$.  Note the effective divisor $$\varphi(D)+\varphi(\theta(D))$$
 belongs to the linear system $|-2K_W|$ and it does not contain the point $q$. So the linear system $|-2K_W|$ is base-point-free and we can find a  smooth surface $S$ from it.
Hence, by Theorem \ref{leethm}, there is a Calabi--Yau threefold $Y$ that covers doubly $W$, branched along $S$ and singularities of $W$.

 Since $H^i(X, \mco_X)=0$ and $H^i(W, \mco_W)=0$ for $i=1,2$, we have isomorphisms
 $$H^2(X, \mbz) \simeq \Pic(X), H^2(W, \mbz) \simeq \Pic(W)$$
 by the exponential sequences.
 Hence we can regard  classes of Cartier divisors of $X, W$ as elements of $H^2(X, \mbz), H^2(W, \mbz)$ respectively.
Let $r$ be the index of Fano threefold $X$ (i.e.\ the largest integer $r$ such that $-K_X = r H_X$ for some ample divisor $H_X$ of $X$).

Now we calculate the invariants of $Y$.
For a double covering with dimension higher than two, it is a non-trivial task to
calculate the topological invariants even in the case that the base of the covering is smooth.
In our case, $S$ is an ample divisor of $W$, so it may be worth trying to apply the Lefschetz hyperplane theorem. However $W$ is not smooth, so the usual  Lefschetz hyperplane theorem does not apply here.
There are other versions of  the Lefschetz hyperplane theorem for singular varieties but they all require that  $W-S$ is smooth, which is not true for our case.
We prove a type of the Lefschetz hyperplane theorem for $S \subset W$.
We say that an element $\alpha$ of an additive Abelian group $G$ is divisible by an integer $k$ if $\alpha = k \alpha'$ for some element $\alpha' \in G$. $\alpha$ is said to be primitive if it is divisibly by only $\pm 1$. We denote the quotient of $G$ by its torsion part as $G_f$.
\begin{lemma} \label{lem1}
The map $H^2(W, \mbq) \ra H^2(S, \mbq)$, induced by the inclusion $S \hookrightarrow W$, is injective and the image $H_W|_S$  in $ H^2(S, \mbz)_f$ of $H_W \in H^2(W, \mbz)$  is divisible by $r$.
\end{lemma}
\begin{proof}
Consider the commutative diagram:

$$
\xymatrix{
 X \ar[r]^{ \varphi}    &   W \\
S_X     \ar[r]^{\varphi|_{S_X}}   \ar@{^{(}->}[u] & S \ar@{^{(}->}[u]
}
$$
where $S_X=\varphi^{-1}(S)$ and the vertical maps are inclusions.
Note
$$\varphi|_{S_X}:S_X \ra S$$
is an unramified double covering.
 We have an induced commutative diagram:
$$
\xymatrix{
 H^2(X, \mbq)   \ar[d] &  H^2 (W, \mbq) \ar[d]  \ar[l]^{ \varphi^*} \\
H^2(S_X, \mbq)        & H^2(S, \mbq)  \ar[l]^{{(\varphi|_{S_X})}^*}
}
$$
Note the pull-backs $\varphi^*:  H^2(W, \mbq)   \ra  H^2 (X, \mbq)$ is injective.
Since $S_X$ is a smooth ample divisor of $X$, the map
$H^2(X, \mbq)   \ra H^2(S_X, \mbq)$
is injective by the Lefschetz hyperplane theorem. So we have the injectivity of the map
$$H^2(W, \mbq) \ra H^2(S, \mbq).$$

Consider another commutative diagram.
$$
\xymatrix{
 H^2(X, \mbz)   \ar[d] &  H^2 (W, \mbz) \ar[d]  \ar[l]^{ \varphi^*} \\
H^2(S_X, \mbz)_f        & H^2(S, \mbz)_f  \ar[l]^{{(\varphi|_{S_X})}^*}
}
$$

Note $-K_X = r H_X$.
We note that $\theta^* (-K_X) = -K_X$ in $H^2(X, \mbz)$. Since $H^2(X, \mbz)$ has no torsion, $\theta^* (H_X) = H_X$ in $H^2(X, \mbz)$ and so
$$(\theta|_{S_X})^*(H_X|_{S_X}) = \theta^*(H_X)|_{S_X} =H_X|_{S_X}$$
 in $H^2(S_X, \mbz)_f$. Hence  $h':=H_X|_{S_X}$ lies in the image of
the map
$$H^2(S, \mbz)_f   \ra H^2(S_X, \mbz)_f.$$

Note $\varphi^*(H_W) = -K_X = r H_X$. Hence
$$(\varphi|_{S_X})^*(H_W|_S)= \varphi^*(H_W)|_{S_X} =rH_X|_{S_X} = r h'$$
 is divisible by $r$ in $H^2(S_X, \mbz)_f$.
Since the map
$$(\varphi|_{S_X})^*:H^2(S, \mbz)_f   \ra H^2(S_X, \mbz)_f$$
is injective,  $H_W|_S$ is divisible by $r$ in $H^2(S, \mbz)_f$.
\end{proof}
We note that $H_W$ is primitive in $H^2(W, \mbz)$. By the above lemma, $H_W|_S$ is not primitive in $H^2(S, \mbz)_f$ when $r > 1$. This is different from what the usual Lefschetz hyperplane theorem expects for smooth threefolds.

\begin{proposition} \label{prop1} We have
$$h^2(Y) \leq h^2(X),$$
$$e(Y) =  e(X)- 24-2(- K_X)^3$$
and
$$\psi^*(H_W) \cdot c_2(Y) = (-K_X)^3 + 24,$$
where $e(Y)$ is the topological Euler characteristic of $Y$ and $c_2(Y)$ is the second Chern class of $Y$.
\end{proposition}
\begin{proof}
Consider the
following fiber product of two double covers:
$$
\xymatrix{
& \widetilde X \ar[dr] \ar[dl]    &   \\
Y     \ar[dr]^{\psi}& & X \ar[dl]_\varphi\\
&W&
}
$$
Then it is easy to see that
\begin{itemize}
\item[(a)] $ \widetilde X   \ra Y$ is an \'etale double cover and
\item[(b)]  $\widetilde X  \ra X$ is the double cover branched along a member $\widetilde S $ of $| -2K_X|$.
\end{itemize}
Using the fact that $\widetilde S$ is an ample divisor of $\widetilde X$, one can show  $h^2(\widetilde X) = h^2(X)$ (\cite{Cy}).
Since $h^2(Y) \leq h^2(\widetilde X)$, we have $h^2(Y) \leq h^2(X).$

Note $e(X) = 2 e(W) - 8$ and $e(S_X) = 2 e(S)$.
Note $S_X \sim -2K_X$ and by the Riemann--Roch theorem,
$$1=\chi(X, \mco_X) = \frac{1}{24} c_2(X) \cdot (-K_X).$$
By the adjunction formula, we have
$$e(S_X) =  c_2(X) \cdot (-2K_X) +4(- K_X)^3 = 48+4(- K_X)^3 .$$
So
$$e(Y) = 2 e(W) - e(S) -8 = e(X)  - \frac{1}{2} e(S_X) = e(X)- 24-2(- K_X)^3.$$

Note $\psi^*(S) \sim 2 S_Y$ and $S \sim 2 H_W$, where $S_Y= \psi^{-1}(S)$. So
$ \psi^*(H_W) \cdot c_2(Y) =S_Y \cdot c_2(Y)$.
By the adjunction formula,
\begin{align*}
S_Y \cdot c_2(Y) &= - S_Y^3 + c_2(S_Y) \\
                 &= -\psi^*(H_W)^3 + e(S_Y) \\
                 &=  -2 H_W^3 + e(S)\\
                 &= -\varphi^*(H_W)^3 + \frac{1}{2} e(S_X) \\
                 &= -(-K_X)^3 + 24+2(- K_X)^3 =  (-K_X)^3 + 24.
\end{align*}
Hence $\psi^*(H_W) \cdot c_2(Y) = (-K_X)^3 + 24$.
\end{proof}

We are interested in the case that the Calabi--Yau threefold $Y$ has Picard number one.
Hence we assume that $X$ has Picard number one. There are four families of them:
\begin{enumerate}
\item[$X_1$:]  complete intersection of a quadric and a quartic in the weighed projective space  $\mbp(1,1,1,1,1,2)$, $r_1=1$, $-K_{X_1}^3=4$, $e(X_1) = -56$.
\item[$X_2$:]  complete intersection of three quadrics in $\mbp^6$, $r_2=1$,  $-K_{X_2}^3=8$, $e(X_2) = -24$.
\item[$X_3$:]  hypersurface of degree $4$ in $\mbp(1,1,1,1, 2)$, $r_3=2$, $-K_{X_3}^3=16$,
 $e(X_3) = -16$.
\item[$X_4$:]  compete intersection of two quadrics in $\mbp^5$, $r_4=2$, $-K_{X_4}^3=32$, 
$e(X_4) = 0$.
\end{enumerate}

\begin{theorem} Suppose that $X$ has Picard number one, then $W$, $Y$ have Picard number one,
$$H_Y^3 = \frac{1}{r^3} (-K_X^3)$$
and
$$ H_Y \cdot c_2(Y)  = \frac{1}{r}((-K_X)^3 + 24),$$
where $H_Y$ is an ample generator of $\Pic(Y)$.
\end{theorem}
\begin{proof}
By Proposition \ref{prop1}, 
$$1 \leq h^2(W) \le  h^2(Y) \le h^2(X)=1,$$ 
so $W$, $Y$ have Picard number one.
Since $Y$ has Picard number one, $\psi^*(H_W) = k H'_Y$ for some ample generator $H'_Y$ of $\Pic(Y)$ ( $\simeq H^2(Y, \mbz)$) and a positive integer $k$. Note that $H_Y-H_Y'$ is a torsion element and that $H_Y'$ is primitive in $H^2(Y, \mbz)$.
We also note that $S_Y$ is a smooth ample divisor of $Y$.
By the Lefschetz hyperplane theorem, $H'_Y|_{S_Y}$ is primitive in $H^2(S_Y, \mbz)_f$.
By Lemma \ref{lem1}, $H_W|_S$ is divisible by $r$ in $H^2(S, \mbz)_f$.
So its image $(\psi|_{S_Y})^*(H_W|_S)$  in $H^2(S_Y, \mbz)_f$ is divisible by $r$.
Note
$$(\psi|_{S_Y})^*(H_W|_S) = \psi^*(H_W)|_{S_Y} =  k (H'_Y|_{S_Y}).$$
So $k$ is divisible by $r$. Let $k=lr$ for some positive integer $l$. We will show that $l=1$.
Note
$$H_Y^3 = {H_Y'}^3 =  \frac{1}{k^3} \psi^*(H_W)^3 = \frac{2}{k^3}  H_W ^3 = \frac{1}{k^3}  \varphi^*(H_W)^3 = \frac{1}{r^3 l^3} (-K_X^3) $$
and
$$ H_Y \cdot c_2(Y) =  \frac{1}{k}\psi^*(H_W) \cdot c_2(Y) =\frac{1}{r l}((-K_X^3) + 24).$$

For $X_1, X_3, X_4$, the condition of $H_Y^3$ being a positive integer requires that $l=1$.
For $X_2$, by the Riemman--Roch theorem, we have
$$\chi(Y, H_{Y}) = \frac{H_{Y}^3}{6} + \frac{H_Y \cdot c_2(Y)}{12} =  \frac{8}{ 6l^3} + \frac{32}{12 l}
=\frac{4+8l^2}{3l^3},$$
which should be an integer. So we have $l=1$ also in this case. Therefore,
$$H_Y^3 = \frac{1}{r^3} (-K_X^3)$$
and
$$ H_Y  \cdot c_2(Y) = \frac{1}{r}((-K_X)^3 + 24).$$

\end{proof}

By Proposition \ref{prop1} and the relation $e(Y) = 2 (h^{1,1}(Y) - h^{1,2}(Y))$, we can determine all the Hodge numbers of $Y$.
We list the invariants of the Calabi--Yau threefolds $Y$'s in Table \ref{t1}. It turns out that they are all new examples . See Appendix I of \cite{Kap} for a list of known examples of Calabi--Yau threefolds of Picard number one.

\begin{table}[ht]
\caption{Invariants of Calabi--Yau double coverings}
\centering
\begin{tabular}{c|cccc}
  \hline\hline
  $X_d$ & $H_Y^3$ & $H_Y \cdot c_2(Y)$ & $h^{1,1}(Y)$ & $h^{1,2}(Y)$ \\
  \hline
  $X_1$ & 4 & 28 & 1 & 45 \\
  $X_2$ & 8 & 32 & 1 & 33 \\
  $X_3$ & 2 & 20 & 1 & 37 \\
 $X_4$ & 4 & 28 & 1 & 45 \\
  \hline
\end{tabular}
\label{t1}
\end{table}

Note that the invariants of $Y_1$ and those of $Y_4$ overlap.
Consider the commutative diagram in the proof of Proposition \ref{prop1}.
For $X_1$, the branch locus of $\widetilde X_1 \ra X_1$ is a quadric section, thus $\widetilde X_1$ is a
$(2, 2, 4)$-weighted complete intersection of $\mbp(1,1,1,1,1,1, 2)$.
For $X_4$, the branch locus of $\widetilde X_4 \ra X_4$ is a quartic section, thus $\widetilde X_4$ is also a $(2, 2, 4)$-weighted complete intersection of $\mbp(1,1,1,1,1,1,2)$.
Therefore, $\widetilde X_1$ and $\widetilde X_4$ are in the same family. Since $Y_1$ and $Y_4$ are \'etale $\mbz_2$-quotients of $\widetilde X_1$ and $\widetilde X_4$ respectively, they have the same invariants.


\end{document}